\documentclass[11pt,   reqno]{amsart}
\usepackage{enumitem,  amsmath, verbatim}
\usepackage[breaklinks, colorlinks]{hyperref} 

\hypersetup{citecolor=blue,          
 linkcolor=blue,          
 urlcolor=blue,
 bookmarksnumbered=true,  
 bookmarksopen=true     
}
\numberwithin{equation}{section}

\theoremstyle{definition}

\newtheorem{theorem}{Theorem}[section]

\newtheorem{lemma}[theorem]{Lemma}
\newtheorem{corollary}{Corollary}[section]

\numberwithin{equation}{section}
\allowdisplaybreaks[4]

\begin{document}

\title[]{Representing Carlitz identity with $q$-Shift Operator}
\author{Yang Dunkun}
\address{Department of Mathematics,   East china Normal University,   500 Dongchuan Roda,   Shanghai 200241,   P. R. China}
\email{Yangdunkun@foxmail.com}
\date{}

\begin{abstract}
  This paper presents a new identity for the $q$-shift operator, building on techniques developed by Liu and Sears. The identity encompasses the Carlitz identity as a special case and utilizes it to extend the Carlitz-type identities proposed by Wang. As applications, we derive an equivalent form of the generalized Carlitz identity to prove two $q$-congruences on cyclotomic polynomials, expanding upon the results of Guo and Wang.
\end{abstract}

\maketitle

\section{introduction}

In this paper, we adhere to the notations introduced in \cite{gasper-book} and assume that $0 < |q| < 1$. The $q$-series and its compact factorials are defined as follows,
\begin{align*}
  (a;q)_0=1, \qquad (a;q)_n=\prod_{k=0}^{n-1}(1-aq^k), \qquad
  (a;q)_{\infty}=\prod_{k=0}^{\infty}(1-aq^k),
\end{align*}
and $(a_1,a_2,\cdots,a_m;q)_n=(a_1;q)_n(a_2;q)_n \cdots (a_m;q)_n$, where $n$ is an integer or $\infty$. The basic hypergeometric series $_{r}\phi_s$ is defined as 
\begin{align}
  _r\phi_s\left(\begin{gathered}
    a_1,a_2,\dots,a_r\\
    b_1,b_2,\dots,b_s
   \end{gathered};q,z\right)=\sum_{n=0}^{\infty}\frac{(a_1,a_2,\dots,a_r;q)_n}{(b_1,b_2,\dots,b_r;q)_n}\left[(-1)^nq^{\binom{n}{2}}\right]^{1\!+\!s\!-\!r}z^n.\notag
\end{align}

The main results of this article consist of two parts. Firstly, we have obtained three transformation identities involving the $q$-shift operator. These identities include the Carlitz identity \cite{carlitz1974} and the Carlitz identity extended by Wang \cite{wmj2008, wmj2013} as special cases, resulting in three more general Carlitz identities. 

Secondly, drawing inspiration from Guo's proof in \cite{gjw2019}, we present an equivalent form for the generalized Carlitz identity discussed in the first part. Utilizing this equivalent form, we establish several $q$-congruences, two of which correspond to Wang's results in \cite{wxx-yml-2023}.

\subsection{$q$-transformation identities of Carlitz}
\

\

In 1974, Carlitz \cite{carlitz1974} obtained the following identity for the three
variables $a$,  $b$, and $q$:
\begin{theorem}\label{carlitz1974-them}
  \begin{equation}\label{Carlitz1974}
    \sum_{k=0}^n \frac{(a,  b ;q)_k}{(q;q)_k}(-a b)^{n-k} q^{\binom{n}{2}\!-\!\binom{k}{2}}
    \!=\!(a;q)_{n\!+\!1} \sum_{k=0}^n \frac{(\!-\!b)^k q^{\binom{k}{2}}}{(q;q)_k(q;q)_{n-k}\left(1\!-\!a q^{n-k}\right)}.
  \end{equation}
\end{theorem}

In 2008 and 2013, Wang \cite{wmj2008, wmj2013} extended the preceding identity by employing advanced mathematical techniques, specifically the Andrews-Askey integral and the $q$-Chu Vandermonde identity. This expansion is presented as follows:
\begin{align}\label{wang-eq1}
  \sum_{k=0}^n \frac{(a,b,bx ; q)_k}{(q ; q)_k}(-ab)^{n-k}q^{\binom{n}{2}\!-\!\binom{k}{2}}\varphi_{n-k}^{(bq^k)}(x|q)\!=\! \frac{(a;q)_{n\!+\!1}}{(q;q)_n} \sum_{k=0}^n\begin{bmatrix}    n \\k
  \end{bmatrix}_q\frac{(\!-\!b)^k q^{\binom{k}{2}}}{1\!-\!aq^{n-k}}\varphi_k^{(b)}(x|q),
\end{align}
and
\begin{align}\label{wang-eq2}
  \sum_{k=0}^n\frac{(a,\frac{1}{u};q)_k}{(q;q)_k}(\!-\!\frac{a}{v})^{n\!-\!k} & q^{\binom{n}{2}\!-\!\binom{k}{2}}\sum_{i=0}^{k}\frac{(q^{\!-\!k},us,ut;q)_iq^i}{(q,uq^{\!-\!k\!+\!1},uvst;q)_i} \  _3\phi_2\left(\begin{gathered}
vs,vt,q^{k\!-\!n}\\
 0,uvstq^i
\end{gathered};q,q\right)\notag \\
& =\frac{(a;q)_{n+1}}{(q;q)_n} \sum_{k=0}^n\begin{bmatrix}
 n \\k
\end{bmatrix}_q\frac{(-\frac{1}{u})^kq^{\binom{k}{2}}}{1-aq^{n\!-\!k}}\  _3\phi_2\left(\begin{gathered}
 vs,vt,q^{\!-\!k}\\
 0,uvst
\end{gathered};q,q\right).
\end{align}
where
\begin{align*}
  \varphi_n^{(b)}(x|q)=\sum_{k=0}^n\begin{bmatrix} n\\k \end{bmatrix}_q(b;q)_k x^k, \qquad
  \begin{bmatrix} n \\k \end{bmatrix}_q
  =\frac{ (q;q)_n }{ (q;q)_k (q;q)_{n-k} },
\end{align*}
are referred to as Hahn polynomials \cite{Hahn1949,Liu2015} and $q$-binomial coefficient, respectively.

In Gasper's book \cite[Ex 1.17]{gasper-book}, the Carlitz identity is
presented to the reader as an exercise, instantly capturing the author's attention
due to its unique structural attributes. Upon encountering this identity, the author
is intellectually stimulated to contemplate the approach for its formal proof. If
both sides of  (\ref{Carlitz1974}) are multiplied by $(q;q)_n/(a;q)_{n+1}$,
their equivalent forms are derived as follows:
\begin{align}\label{carlitz1974-equ1}
  \frac{(q;q)_n}{(a;q)_{n+1}}\sum_{k=0}^n \frac{(a,  b ;q)_k}{(q;q)_k}(-a b)^{n-k} q^{(n\!-\!k)(n\!+\!k\!-\!1) / 2}
  \!= \sum_{k=0}^n \begin{bmatrix}
          n \\k
        \end{bmatrix}_q\frac{(\!-\!b)^k q^{\binom{k}{2}}}{1\!-\!a q^{n-k}},
\end{align}
commencing from the left-hand side of (\ref{carlitz1974-equ1}), we deduce the right-hand side of
(\ref{carlitz1974-equ1}) utilizing the iterative approach and the $q$-binomial theorem.
Consequently, the proof of Theorem \ref{carlitz1974-them} is established.
Subsequently, leveraging the $q$-binomial coefficient present on the
right-hand side of (\ref{carlitz1974-equ1}), Liu \cite{liuzhiguo-2023}
and the present author \cite{ydk-2023} have developed operator representations
for $q$-polynomials containing such coefficients. Building upon the foundational
work of Sears \cite{sears-1951} and Wang \cite{wmj2008, wmj2013}, we derive the subsequent theorem.

\begin{theorem}\label{carlitz-general-1}
  Let $f(a)$ be any function of $a$. Then
  \begin{align}
     & \sum_{k=0}^n\begin{bmatrix}
          n \\k
        \end{bmatrix}_qb^{n\!-\!k}(b;q)_{k}\eta_a^{k}\Delta_a^{n\!-\!k} f(a)\!=\!\sum_{k=0}^n\begin{bmatrix}
      n \\k
    \end{bmatrix}_q(-b)^kq^{\binom{k}{2}}f(aq^{n-k}),\label{carlitz-cor-1}    \\
     & \sum_{k=0}^n\begin{bmatrix}
          n \\k
        \end{bmatrix}_qb^{n\!-\!k}(b,bz;q)_{k}\varphi_{n\!-\!k}^{(bq^k)}(z|q)\eta_a^{k}\Delta_a^{n-k}f(a)\!=\!\sum_{k=0}^n\begin{bmatrix}
  n \\k
\end{bmatrix}_q(\!-\!b)^kq^{\binom{k}{2}}\varphi_k^{(b)}(z|q)f(aq^{n\!-\!k}),\label{carlitz-cor-11} \\
     & \sum_{k=0}^n\begin{bmatrix}
          n \\k
        \end{bmatrix}_q(1/u;q)_kv^{k\!-\!n}\sum_{i=0}^{k}\frac{(q^{\!-\!k},us,ut;q)_iq^i}{(q,uq^{\!-\!k\!+\!1},uvst;q)_i}\  _3\phi_2\left(\begin{gathered}
     vs,vt,q^{k\!-\!n}\\
     0,uvstq^i
   \end{gathered};q,q\right)\eta_a^{k}\Delta_a^{n\!-\!k} f(a)\notag\\
     & \qquad\qquad\qquad\qquad\!=\!\sum_{k=0}^n\begin{bmatrix}
      n \\k
    \end{bmatrix}_q(\!-\!1/u)^kq^{\binom{k}{2}}\ _3\phi_2\left(\begin{gathered}
    us,vt,q^{\!-\!k}\\0,uvst
  \end{gathered};q,q\right)f(aq^{n-k}).\label{carlitz-cor-111}
  \end{align}

  Here, $\Delta_a^k = (\eta_a-1)(\eta_a-q)\dots(\eta_a-q^{k-1})$,
  where $\eta_a$ denotes the $q$-shift operator acting on the variable $a$,
  defined as $\eta_af(a)=f(aq)$.
\end{theorem}

The operator $\Delta_a^k$ captures the product of consecutive differences arising from the application of the $q$-shift operator. Such an operator is instrumental in elucidating certain $q$-analog expressions. Let $f(a)=\frac{1}{1-a}$ in (\ref{carlitz-cor-1}), which serves as an alternative proof for Theorem \ref{carlitz1974-them}. Additionally, by letting $f(a)=\frac{1}{1-a}$ in both (\ref{carlitz-cor-11}) and (\ref{carlitz-cor-111}), we obtain (\ref{wang-eq1}) and (\ref{wang-eq2}). More generally, if $f(a)=\frac{(ay;q)_{\infty}}{(ax;q)_{\infty}}$, we derive the following theorem.

\begin{theorem}\label{carlitz-general-2}
  Let $n$ be a positive integer and $P_k(x,y)=(x-y)(x-yq)\cdots(x-yq^{k-1})$. Then
  \begin{align}
     & \sum_{k=0}^n\begin{bmatrix}
          n \\k
        \end{bmatrix}_q(ax, b;q)_kP_{n\!-\!k}(x, y)(\!-\!ab)^{n\!-\!k}q^{\binom{n}{2}\!-\!\binom{k}{2}}\notag          \\
     & \qquad\qquad\qquad\qquad\qquad=  (ay;q)_n \sum_{k=0}^n\begin{bmatrix}
        n \\k
      \end{bmatrix}_q\frac{(ax;q)_{n\!-\!k}}{(ay;q)_{n\!-\!k}}(\!-\!b)^kq^{\binom{k}{2}},\label{carlitz-cor-2}         \\
     & \sum_{k=0}^n\begin{bmatrix}
          n \\k
        \end{bmatrix}_q(b,ax, bz;q)_kP_{n\!-\!k}(x, y)(\!-\!ab)^{n\!-\!k}q^{\binom{n}{2}\!-\!\binom{k}{2}}\varphi_{n-k}^{(bq^k)}(z|q),\notag \\
     & \qquad\qquad\qquad\qquad\qquad=(ay;q)_n \sum_{k=0}^n\begin{bmatrix}
      n \\k
    \end{bmatrix}_q\frac{(ax;q)_{n\!-\!k}}{(ay;q)_{n\!-\!k}}(\!-\!b)^kq^{\binom{k}{2}}\varphi_k^{(b)}(z|q),\label{carlitz-cor-22} \\
     & \sum_{k=0}^n\begin{bmatrix}
          n \\k
        \end{bmatrix}_q(1/u,ax;q)_kP_{n\!-\!k}(x, y)(\!-\!a/v)^{n\!-\!k}q^{\binom{n}{2}\!-\!\binom{k}{2}}\notag   \\
     & \qquad\qquad\qquad\qquad\times\sum_{i=0}^{k}\frac{(q^{\!-\!k},us,ut;q)_iq^i}{(q,uq^{\!-\!k\!+\!1},uvst;q)_i} \  _3\phi_2\left(\begin{gathered}
    vs,vt,q^{k\!-\!n}\\
    0,uvstq^i
  \end{gathered};q,q\right)\notag          \\
     & \qquad\qquad=(ay;q)_n \sum_{k=0}^n\begin{bmatrix}
n \\k
        \end{bmatrix}_q\frac{(ax;q)_{n\!-\!k}}{(ay;q)_{n\!-\!k}}(\!-\!1/u)^kq^{\binom{k}{2}}\  _3\phi_2\left(\begin{gathered}
      vs,vt,q^{\!-\!k}\\
      0,uvst
    \end{gathered};q,q\right).\label{carlitz-cor-222}
  \end{align}
\end{theorem}

Let $x=1$ and $y=q^m$, where $m\geq1$, in (\ref{carlitz-cor-2}). We obtain the generalized Carlitz identity as follows:
\begin{align}\label{carlitz-m}
  \sum_{k=0}^n\frac{(a,b;q)_k(q^m;q)_{n-k}}{(q;q)_k(q;q)_{n-k}}(\!-\!ab)^{n-k}q^{\binom{n}{2}\!-\!\binom{k}{2}}
  \!=\!(a;q)_{n\!+\!m} \sum_{k=0}^n \frac{(\!-\!b)^kq^{\binom{k}{2}}}{(q;q)_k(q;q)_{n-k}(aq^{n-k};q)_{m}},
\end{align}
when $m=1$, it reduces to (\ref{Carlitz1974}). Similarly, let $x=1$ and $y=q^m$, where $m \geq 1$, in (\ref{carlitz-cor-22}) and (\ref{carlitz-cor-222}), we also obtain their generalized forms.

\subsection{Some $q$-Congruences by generalized Carlitz identity}
\

\

The definition of the $n$-th cyclotomic polynomial is reviewed as follows:
\begin{align*}
  \Phi_{n}(q)=\prod_{\substack{1\leq k\leq n,\\ \gcd(n, k)=1}}(q-\zeta^k),
\end{align*}
where $\zeta$  is an $ n$-th primitive root of unity. In 2019, 
Guo \cite{gjw2019} substantiated
Tauraso's congruence conjecture \cite{rt2013} through the utilization of (\ref{Carlitz1974}) as follows:
\begin{align}
  \sum_{k=0}^{n-1}\frac{q^k}{(-q;q)_k}\begin{bmatrix}
       2k \\k
     \end{bmatrix}
  \equiv (\!-\!1)^{\frac{n\!-\!1}{2}}q^{\frac{n^2-1}{4}} \mod \Phi_n(q)^2,
\end{align}
Subsequently, this particular identity has remained a prevalent choice for establishing various $q$-congruences, as evidenced by its continued application in recent works such as \cite{gcy-gjw-2021, gjw-2023, wc-nhx-2022, wxx-2022}.

Drawing inspiration from Guo's proof in \cite{gjw2019}, we have derived the following equivalent forms of the generalized Carlitz identity (\ref{carlitz-m}).
\begin{theorem}\label{app-lemma-1}
  Let $n$ be a positive odd integer. Then
  \begin{align}\label{app-lemma-equ-1}
    \sum_{k=0}^{n-1} & \frac{(a, b;q^2)_k}{(q^2;q^2)_k}(q^{2n\!-\!2k};q^2)_{m\!-\!1}(\!-\!ab)^{n\!-\!k\!-\!1}q^{-k^2\!+\!k}\notag        \\
          & =\begin{bmatrix}
    n\!-\!1 \\\frac{n\!-\!1}{2}
  \end{bmatrix}_{q^2}\begin{bmatrix}
 2n\!-\!1 \\n\!-\!1
\end{bmatrix}_q\frac{(\!-\!1)^{\frac{n-1}{2}}q^{\!-\!\frac{(n\!-\!1)(3n\!-\!5)}{4}}}{(-q;q)_{n\!-\!1}^2}
    \frac{(q^2;q^2)_{m\!-\!1}(a;q^2)_{n\!+\!m\!-\!1}}{(q;q^2)_n}\notag\\
          & \ \ \ \ \ \  \ \ \ \ \times (1\!-\!q^n)\!\sum_{k=\!-\!\frac{n\!-\!1}{2}}^{\frac{n\!-\!1}{2}}
    \frac{(q^{1\!-\!n};q^2)_k}{(q^{1\!+\!n};q^2)_k}
    \frac{b^{k\!+\!\frac{n\!-\!1}{2}}q^{2nk\!-\!2k}}{(aq^{n\!-\!2k-1};q^2)_m}.
  \end{align}
\end{theorem}

Compared to  (\ref{carlitz-m}),  (\ref{app-lemma-equ-1}) may appear more verbose and complex. However, when combined with the works of Guo \cite{GJW-2018} and Pan \cite{Ljx-Ph-Zy-2015,Panhao-2007},  (\ref{app-lemma-equ-1}) provides a viable pathway for deriving $q$-congruences under the$ \mod\Phi_n(q)^2$ condition. Therefore, the $q$-congruences in our article are derived based on (\ref{app-lemma-equ-1}).

\begin{theorem}\label{app-th1}
  Let $ n$  be a positive odd integer,  $ 1\leq m\leq\frac{n+1}{2}$. Then
  \begin{align}\label{app-th1-equ}
    \sum_{k=0}^{n-m}\frac{(q^{4m\!-\!2};q^4)_k}{(q^2;q^2)_k}q^{3k\!-\!k^2\!-\!4km}\equiv
    (-1)^{\frac{n\!-\!1}{2}\!+\!m\!-\!1}q^{\!-\!\frac{n^2\!-\!1}{2}\!+\!2m^2\!-\!2m}\mod\Phi_n(q).
  \end{align}
\end{theorem}

Obviously, the left-hand side of (\ref{app-th1-equ}) is an even function in $q$. This implies that (\ref{app-th1-equ}) is also valid modulo $\Phi_n(-q)$. Therefore, (\ref{app-th1-equ}) holds modulo $\Phi_n(q^2)$ because $\Phi_n(q^2)=\Phi_n(q)\Phi_n(-q)$. By letting $q^2\to q$ and subsequently letting $q\to q^{-1}$ in (\ref{app-th1-equ}), and applying similar methods as in \cite[Lemma 2.1]{wxx-2022} and \cite[Lemma 2.1, Lemma 3.1]{wc-nhx-2022}, we can derive the following corollary.

\begin{corollary}\label{app-th1-equ-cor}
  Let $ n$  be a positive odd integer,  $ 1\leq m\leq\frac{n+1}{2}$. Then
  \begin{align}
    \sum_{k=0}^{n-m}\frac{(q^{2m\!-\!1};q^2)_k}{(q;q)_k}q^{k}\equiv
    (-1)^{\frac{n\!-\!1}{2}\!+\!m\!-\!1}q^{\frac{n^2\!-\!1}{4}\!-\!m^2\!+\!m}\mod\Phi_n(q),
  \end{align}
  \begin{align}
    \sum_{k=0}^{n-m}\frac{(q^{2m\!-\!1};q^2)_k}{(q;q)_k}q^{2k}\equiv
    (-1)^{\frac{n\!-\!1}{2}\!+\!m\!-\!1}q^{\frac{n^2\!-\!1}{4}\!-\!m^2\!-\!m\!+\!1}\mod\Phi_n(q).
  \end{align}
\end{corollary}
Let $q\rightarrow 1$ in the above two congruences. We can find the following corollary.
\begin{corollary}
  Let $p$ be an odd prime. Then
  \begin{align}
    \sum_{k=0}^{p-m}\frac{2^k(m-\frac{1}{2})_k}{k!}\equiv (-1)^{\frac{p\!-\!1}{2}\!+\!m\!-\!1}\mod p.
  \end{align}
\end{corollary}

Particularly, when $m=1$, the mentioned congruence corresponds to the following congruence, as proven by Sun and Tauraso \cite{s-t-2010}.
\begin{align}
  \sum_{k=0}^{p^r-1}\frac{1}{2^k}\binom{2k}{k}\equiv (-1)^{\frac{n\!-\!1}{2}}\mod p.
\end{align}

In 2022,  Wang and Yu  \cite{wxx-2022} proved the following $q$-congruences,
$$\sum_{k=0}^{n-1}\frac{(q^{4d\!+\!2};q^4)_k}{(q^2;q^2)_k}q^{\!-\!4dk\!-\!k^2\!-\!k}
  \equiv (\!-\!1)^{\frac{n-1}{2}+d}q^{-\frac{n^2-1}{2}+2d^2+2d} \mod\Phi_n(q),$$

  Subsequently, Wang and Yu proved that the aforementioned $q$-congruence holds modulo $\Phi_n(q)^2$, as follows.

\begin{align}
  \sum_{k=0}^{n-1}  \frac{(q^{2d\!+\!1};q^2)_k}{(q;q)_k}q^{k}
    &\equiv (\!-\!1)^{\frac{n-1}{2}\!+\!d}q^{\frac{n^2\!-\!(2d\!+\!1)^2}{4}}\!+\!\mathbf{Sgn}(d)\sum_{t=1}^{|d|}\frac{(\!-\!1)^{d\!-\!t}(1\!+\!q^{2t\!-\!1})[n]}{q^{2\binom{d\!+\!1}{2}\!-\!2\binom{t}{2}}[2t\!-\!1]},\label{wang-yu-1}\\
    \sum_{k=0}^{n-1}\frac{(q^{2d\!+\!1};q^2)_k}{(q;q)_k}q^{2k}
    &\equiv (\!-\!1)^{\frac{n-1}{2}\!+\!d}q^{\frac{n^2\!-\!(2d+1)(2d+5)}{4}}
    \!-\!\frac{1\!-\!q}{q^{2d\!+\!1}}[n]\notag\\
    &\qquad\qquad\qquad\qquad+\mathbf{Sgn}(d)\sum_{t=1}^{|d|}\frac{(\!-\!1)^{d\!-\!t}(1\!+\!q^{2t\!-\!1})[n]}{q^{d^2\!+\!3d\!+\!1\!-\!2\binom{t}{2}}[2t\!-\!1]},\label{wang-yu-2}
\end{align}
where $[n]$ is the $q$-integer defined as $[n]=\frac{1-q^n}{1-q}$, and 
$\mathbf{Sgn}(d)$ is the Sign function defined as
$$ \mathbf{Sgn}(d)=\begin{cases}
  1,  & d>0; \\
  0,  & d=0; \\
  -1, & d<0.
\end{cases}$$
By using (\ref{app-lemma-equ-1}), we obtain the following theorem, which provides the equivalent forms of (\ref{wang-yu-1}) and (\ref{wang-yu-2}).
\begin{theorem}\label{app-th2}
  Let $ n$  be a positive odd integer, $ d \in\{-\frac{n-3}{2}, -\frac{n-5}{2},
    \dots, \frac{n-5}{2}, \frac{n-3}{2}\}$.  Then
  \begin{align}\label{app-th2-equ}
    \sum_{k=0}^{n-1} & \frac{(q^{4d\!+\!2};q^4)_k}{(q^2;q^2)_k}q^{\!-\!4dk\!-\!k^2\!-\!k}
    \equiv (\!-\!1)^{\frac{n-1}{2}\!+\!d}q^{-\frac{n^2-1}{2}+2d^2+2d}\!+\!\mathbf{Sgn}(d)
    (1\!-\!q^n)q^{2d^2\!+\!3d}\times\notag      \\
          & \sum_{j=1}^{|d|}\frac{q^{j\!-\!2dj}(1\!+\!q^{(4j\!-\!2)d})}{1\!-\!q^{4j\!-\!2}}
    \{1\!+\!(\!-\!1)^{j\!+\!d}\!+\!q^{2j\!-\!1}[(\!-\!1)^{j\!+\!d}\!-\!1]\}\mod\Phi_n(q)^2.
  \end{align}
\end{theorem}

Obviously, the left-hand side of (\ref{app-th2-equ}) is an even function in $q$. Using a similar approach to Corollary \ref{app-th1-equ-cor}, by letting $q^2\to q$ and then letting $q\to q^{-1}$ in (\ref{app-th2-equ}), we can derive the following corollary, which is the equivalent form of (\ref{wang-yu-1}).

\begin{corollary}
  Let $ n$  be a positive odd integer, $ d \in\{-\frac{n-3}{2}, -\frac{n-5}{2},
  \dots, \frac{n-5}{2}, \frac{n-3}{2}\}$.  Then
  \begin{align}\label{app-th2-equ-q-k}
    \sum_{k=0}^{n-1} & \frac{(q^{2d\!+\!1};q^2)_k}{(q;q)_k}q^{k}
    \equiv (\!-\!1)^{\frac{n-1}{2}\!+\!d}q^{\frac{n^2-1}{4}-d^2-d}\!+\!\mathbf{Sgn}(d)
    (1\!-\!q^n)\times\notag \\
          & \sum_{j=1}^{|d|}\frac{q^{jd\!+\!2j\!-\!d^2\!-\!1-\!\frac{j\!+\!3d}{2}}(1\!+\!q^{-(2j\!-\!1)d})}{1\!-\!q^{2j\!-\!1}}
    \frac{\{1\!+\!(\!-\!1)^{j\!+\!d}\!+\!q^{\frac{1}{2}\!-\!j}[(\!-\!1)^{j\!+\!d}\!-\!1]\}}{2} \mod\Phi_n(q)^2.
  \end{align}
\end{corollary}

Similarly, using \cite[(2.11)]{wxx-2022} and Lemma \ref{app-th2-lamma2}, we arrive at the following congruence:
\begin{align}
  \sum_{k=0}^{n-1}\frac{(q^{2d\!+\!1};q^2)_k}{(q;q)_k}q^{2k}\equiv
  q^{-2d-1}\sum_{k=0}^{n-1}\frac{(q^{2d\!+\!1};q^2)_k}{(q;q)_k}q^{k}-q^{-2d-1}(1-q^n) \mod\Phi_n(q)^2,
\end{align}
Substituting (\ref{app-th2-equ-q-k}) into the above $q$-congruence, we obtain the following corollary, which is the equivalent form of (\ref{wang-yu-2}).
\begin{corollary}
  Let $ n$  be a positive odd integer, $ d \in\{-\frac{n-3}{2}, -\frac{n-5}{2},
  \dots, \frac{n-5}{2}, \frac{n-3}{2}\}$.  Then
  \begin{align}\label{app-th2-equ-q-2k}
    \sum_{k=0}^{n-1} & \frac{(q^{2d\!+\!1};q^2)_k}{(q;q)_k}q^{2k}
    \equiv (\!-\!1)^{\frac{n-1}{2}\!+\!d}q^{\frac{n^2\!-\!1}{4}\!-\!d^2\!-\!3d\!-\!1}
    \!-\!q^{\!-\!2d\!-\!1}(1-q^n)+\mathbf{Sgn}(d)
    (1\!-\!q^n)\times\notag \\
          & \sum_{j=1}^{|d|}\frac{q^{jd\!+\!2j\!-\!d^2\!-\!2-\!\frac{j\!+\!7d}{2}}(1\!+\!q^{-(2j\!-\!1)d})}{1\!-\!q^{2j\!-\!1}}
    \frac{\{1\!+\!(\!-\!1)^{j\!+\!d}\!+\!q^{\frac{1}{2}\!-\!j}[(\!-\!1)^{j\!+\!d}\!-\!1]\}}{2} \mod\Phi_n(q)^2.
  \end{align}
\end{corollary}

Particularly, by setting $d=-1$ in both (\ref{app-th2-equ-q-k}) and
(\ref{app-th2-equ-q-2k}), we reach the Wang-Ni $q$-congruence as established
in \cite[Theorem 1.1]{wc-nhx-2022}.

\begin{align}
   & \sum_{k=0}^{n-1}\frac{(q^{\!-\!1};q^2)_k}{(q;q)_k}q^{k}
  \equiv (\!-\!1)^{\frac{n+1}{2}}q^{\frac{n^2\!-\!1}{4}}-(1+q)[n] \mod\Phi_n(q)^2, \\
   & \sum_{k=0}^{n-1}\frac{(q^{\!-\!1};q^2)_k}{(q;q)_k}q^{2k}
  \equiv (\!-\!1)^{\frac{n+1}{2}}q^{\frac{n^2\!+\!3}{4}}-2q[n] \mod\Phi_n(q)^2.
\end{align}
Similarly,  setting $ d=0$  in (\ref{app-th2-equ-q-2k}) and setting $ d=1$  in (\ref{app-th2-equ-q-k}) we
arrive at the Wang-Ni $q$-congruence \cite[Theorem 1.4]{wc-nhx-2022} as follows,
\begin{align}
   & \sum_{k=0}^{n-1}\frac{(q;q^2)_k}{(q;q)_k}q^{2k}
  \equiv (\!-\!1)^{\frac{n-1}{2}}q^{\frac{n^2\!-\!5}{4}}+\frac{q-1}{q}[n] \mod\Phi_n(q)^2, \\
   & \sum_{k=0}^{n-1}\frac{(q^{3};q^2)_k}{(q;q)_k}q^{k}
  \equiv (\!-\!1)^{\frac{n+1}{2}}q^{\frac{n^2\!-\!9}{4}}+\frac{1+q}{q^2}[n] \mod\Phi_n(q)^2. \label{gjw-con}
\end{align}
Here, (\ref{gjw-con}) refers to a conjecture put forth by Gu and
Guo \cite{gcy-gjw-2021}. By letting $q\rightarrow1$ in either
(\ref{app-th2-equ-q-k}) or (\ref{app-th2-equ-q-2k}), we derive the following congruence.
\begin{corollary}
  Let $d$ be an integer and $p$ an odd prime such that $p^r > 2|d| + 1$. Then
  \begin{align}
    \sum_{k=0}^{p^r-1}\frac{2^k(d+\frac{1}{2})_k}{k!}
    \equiv (-1)^{\frac{p^r\!-\!1}{2}\!+\!d}\!+\!
    \mathbf{Sgn}(d)\sum_{j=1}^{|d|}\frac{2p^r(\!-\!1)^{j\!+\!d}}{2j\!-\!1}\mod p^2.
  \end{align}
  Similarly, when $d=0$, the congruence mentioned above
  corresponds to the following congruence, as proven by Sun\cite{szw2010},
  \begin{align}
    \sum_{k=0}^{p^r-1}\frac{1}{2^k}\binom{2k}{k}\equiv (-1)^{\frac{n\!-\!1}{2}}\mod p^2.
  \end{align}
\end{corollary}

\section{The  proofs of Theorem\ref{Carlitz1974}-Theorem\ref{carlitz-general-2}}

\subsection{The new proof of  Theorem\ref{carlitz1974-them}}

\begin{proof}
  Let  \begin{align}\label{carlitz-l-phi}
    \phi_n(a, b)=\frac{(q;q)_n}{(a;q)_{n+1}}
    \sum_{k=0}^n\frac{(a, b;q)_k}{(q;q)_k}(-ab)^{n-k}q^{\binom{n}{2}-\binom{k}{2}},
  \end{align}
  and by
  \begin{align}
    \frac{(q;q)_n}{(q;q)_k}=
    \frac{(q;q)_{n-1}}{(q;q)_{k-1}}+\frac{(q;q)_{n-1}}{(q;q)_{k}}q^k(1-q^{n\!-\!k}), \notag
  \end{align}
  then, we have calculated that
  \begin{align}
    \phi_n(a, b) & \!=\!\frac{1}{(a;q)_{n\!+\!1}}
    \sum_{k=0}^n\left(\frac{(q;q)_{n\!-\!1}}{(q;q)_{k\!-\!1}}\!+\!\frac{(q;q)_{n\!-\!1}}{(q;q)_{k}}q^k(1\!-\!q^{n\!-\!k})\right)
    (a, b;q)_k(\!-\!ab)^{n-k}q^{\binom{n}{2}\!-\!\binom{k}{2}}\notag  \\
      & =\!\frac{1}{(a;q)_{n\!+\!1}}\sum_{k=0}^{n-1}\frac{(q;q)_{n\!-\!1}}{(q;q)_{k\!-\!1}}
    (a, b;q)_{k+1}(\!-\!ab)^{n-k-1}q^{\binom{n-1}{2}\!-\!\binom{k}{2}\!+\!n\!-\!1\!-\!k}\notag          \\
      & \ \ \ \ \ \ \ -\!\frac{1}{(a;q)_{n}}\sum_{k=0}^{n-1}
    \frac{(q;q)_{n\!-\!1}}{(q;q)_{k}}(a, b;q)_k\frac{abq^k-abq^n}{1-aq^n}(\!-\!ab)^{n\!-\!k\!-\!1}q^{\binom{n}{2}\!-\!\binom{k}{2}}\notag\\
      & =\!\frac{(q;q)_{n-1}}{(aq;q)_{n}}\sum_{k=0}^{n-1}\frac{(aq;q)_{k}(b;q)_{k+1}}{(q;q)_{k}}
    (\!-\!abq)^{n-k-1}q^{\binom{n-1}{2}\!-\!\binom{k}{2}}\notag       \\
      & \ \ \ \ \ \ \ -\!\frac{(q;q)_{n-1}}{(a;q)_{n}}\sum_{k=0}^{n-1}
    \frac{(a, b;q)_{k}}{(q;q)_{k}}\frac{(1\!-\!aq^n)b\!-\!b(1\!-\!aq^k)}{1-aq^n}(\!-\!ab)^{n\!-\!k\!-\!1}q^{\binom{n-1}{2}\!-\!\binom{k}{2}\!+\!n\!-\!1}\notag \\
      & =\!\frac{(q;q)_{n-1}}{(aq;q)_{n}}\sum_{k=0}^{n-1}\frac{(aq;q)_{k}(b;q)_{k+1}}{(q;q)_{k}}
    (\!-\!abq)^{n-k-1}q^{\binom{n-1}{2}\!-\!\binom{k}{2}}\notag       \\
      & \ \ \ \ \ \ \ -\!bq^{n-1}\frac{(q;q)_{n-1}}{(a;q)_{n}}\sum_{k=0}^{n-1}\frac{(a, b;q)_{k}}{(q;q)_{k}}
    (\!-\!ab)^{n-k-1}q^{\binom{n-1}{2}\!-\!\binom{k}{2}}\notag        \\
      & \ \ \ \ \ \ \ +\!\frac{(q;q)_{n-1}}{(aq;q)_{n}}\sum_{k=0}^{n-1}\frac{(aq, b;q)_{k}}{(q;q)_{k}}
    (\!-\!abq)^{n-k-1}q^{\binom{n-1}{2}\!-\!\binom{k}{2}}bq^k\notag   \\
      & =\!\frac{(q;q)_{n-1}}{(aq;q)_{n}}\sum_{k=0}^{n-1}\frac{(aq, b;q)_{k}}{(q;q)_{k}}
    (\!-\!abq)^{n-k-1}q^{\binom{n-1}{2}\!-\!\binom{k}{2}}(1-bq^k+bq^k)\notag     \\
      & \ \ \ \ \ \ \ -\!bq^{n-1}\frac{(q;q)_{n-1}}{(a;q)_{n}}\sum_{k=0}^{n-1}\frac{(a, b;q)_{k}}{(q;q)_{k}}
    (\!-\!abq)^{n-k-1}q^{\binom{n-1}{2}\!-\!\binom{k}{2}}\notag       \\
      & =\phi_{n-1}(aq, b)-bq^{n-1}\phi_{n-1}(a, b)\notag  \\
      & =(\eta_a-bq^{n-1})\phi_{n-1}(a, b), \notag
  \end{align}
  so
  $$\phi_n(a, b)=(\eta_a-bq^{n-1})\phi_{n-1}(a, b), $$
  repeated iteration, known from $q$-binomial theorem  \cite[Ex 1.2 (vi)]{gasper-book}
  \begin{align}
    \phi_n(a, b) & =(\eta_a-bq^{n-1})\phi_{n-1}(a, b)\notag        \\
      & =(\eta_a-bq^{n-1})(\eta_a-bq^{n-2})\phi_{n-2}(a, b)\notag  \\
      & =\dots\notag  \\
      & =(\eta_a-bq^{n-1})(\eta_a-bq^{n-2})\dots(\eta_a-b)\phi_{0}(a, b)\notag\\
      & =\sum_{k=0}^n\begin{bmatrix}
 n \\k
          \end{bmatrix}_q(-b)^kq^{\binom{k}{2}}\eta_a^{n-k}\phi_{0}(a, b)\label{proof-1-phi-expan},
  \end{align}
  substituting $ \phi_{0}(a, b)=\frac{1}{1-a}$  into the above equation yields
  \begin{align}\label{Carlitz1974_rhs}
    \phi_n(a, b)=\sum_{k=0}^n\begin{bmatrix}
      n \\k
    \end{bmatrix}_q\frac{(-b)^kq^{\binom{k}{2}}}{1-aq^{n-k}}.
  \end{align}
  Therefore,  equation (\ref{Carlitz1974}) holds.
\end{proof}

\subsection{The proof of Theorem \ref{carlitz-general-1}}
\ \

\ \

In 2022, Liu \cite{liuzhiguo-2023} presented
the expression of the Rogers-Szegö polynomial using the $q$-shift operator as follows:
$$ \sum_{k=0}^n\begin{bmatrix}
    n \\k
  \end{bmatrix}_qx^k=(x+\eta_x)^n\cdot1,$$
and the author also employed a similar approach to represent a class of
equations using the $q$-shift operator in \cite{ydk-2023}. Considering that
the polynomial in equation (\ref{Carlitz1974_rhs}) incorporates the $q$-binomial
coefficient $\begin{bmatrix}
    n \\k
  \end{bmatrix}_q$,  it leads us to consider that an operator represents $\phi_n(a, b)$,
Exploring the potential application of the $q$-Pascal rule
\cite{KVPC-book} could provide a way to establish this representation for $\phi_n(a, b)$
as an operator.

\begin{lemma}
  Let \begin{align*}
    \phi_n(a, b)=\sum_{k=0}^n\begin{bmatrix}
      n \\k
    \end{bmatrix}_q\frac{(-b)^kq^{\binom{k}{2}}}{1-aq^{n-k}}.
  \end{align*}
  Then
  \begin{equation}\label{carlitz-r-phi}
    \phi_n(a, b)=(\eta_a\eta_b-b\eta_b)^n\cdot\frac{1}{1-a}.
  \end{equation}
\end{lemma}

\begin{proof}
  By $q$-Pascal rule \cite{KVPC-book},
  $$\begin{bmatrix}
      n\!+\!1 \\k
    \end{bmatrix}_q\!=\!\begin{bmatrix}
      n \\k
    \end{bmatrix}_qq^k\!+\!\begin{bmatrix}
      n \\k\!-\!1
    \end{bmatrix}_q,  $$
  we have
  \begin{align}
    \phi_n(a, b) & =\sum_{k=0}^n\begin{bmatrix}
 n \\k
          \end{bmatrix}_q\frac{(-b)^kq^{\binom{k}{2}}}{1-aq^{n-k}}\notag\\
      & =\sum_{k=0}^{n}\left(\!\begin{bmatrix}
  n-1 \\k
\end{bmatrix}_qq^k\!+\!\begin{bmatrix}
   n-1 \\k\!-\!1
 \end{bmatrix}_q\right)\frac{(-b)^kq^{\binom{k}{2}}}{1-aq^{n-k}}\notag    \\
      & =\sum_{k=0}^{n-1}\begin{bmatrix}
     n-1 \\k
   \end{bmatrix}_q\frac{(-bq)^kq^{\binom{k}{2}}}{1-aq^{n-k}}+\sum_{k=1}^{n}\begin{bmatrix}
          n-1 \\k-1
        \end{bmatrix}_q\frac{(-b)^kq^{\binom{k}{2}}}{1-aq^{n-k}}\notag \\
      & =\sum_{k=0}^{n-1}\begin{bmatrix}
     n-1 \\k
   \end{bmatrix}_q\frac{(-bq)^kq^{\binom{k}{2}}}{1-aq\cdot q^{n-k-1}}-
    b\sum_{k=0}^{n-1}\begin{bmatrix}
 n-1 \\k
          \end{bmatrix}_q\frac{(-bq)^kq^{\binom{k}{2}}}{1-aq^{n-k-1}}\notag       \\
      & =(\eta_a\eta_b-b\eta_b)\phi_{n-1}(a, b), \notag
  \end{align}
  through repeated iteration, we can derive
  \begin{align}
    \phi_n(a, b)=(\eta_a\eta_b-b\eta_b)^n\phi_{0}(a, b)=(\eta_a\eta_b-b\eta_b)^n\cdot\frac{1}{1-a}.\notag
  \end{align}

\end{proof}

Obviously,  since $\eta_a\eta_b\cdot b\eta_b=qb\eta_b\cdot \eta_a\eta_b$ ,
by  \cite[Ex 1.35]{gasper-book}
\begin{align}
  (\eta_a\eta_b-b\eta_b)^n\cdot\frac{1}{1-a} & =\sum_{k=0}^n\begin{bmatrix}
       n \\k
     \end{bmatrix}_q(-b\eta_b)^k\eta_a^{n-k}\eta_b^{n-k}\cdot\frac{1}{1-a}\notag         \\
 & =\sum_{k=0}^n\begin{bmatrix}
       n \\k
     \end{bmatrix}_q(-b)^kq^{\binom{k}{2}}\eta_a^{n-k}\eta_b^{n}\cdot\frac{1}{1-a}\notag \\
 & =\sum_{k=0}^n\begin{bmatrix}
       n \\k
     \end{bmatrix}_q\frac{(-b)^kq^{\binom{k}{2}}}{1-aq^{n-k}}.\notag
\end{align}
However, this approach does not yield the equation form (\ref{carlitz-l-phi}).
Our interest lies in understanding how equation (\ref{carlitz-l-phi}) is derived
without prior knowledge of its expression. Using equation (\ref{Carlitz1974_rhs})
to deduce equation (\ref{carlitz-l-phi}) is the question at hand. An operator
can directly represent equation (\ref{carlitz-l-phi}). By employing Sears' method
as outlined in \cite{sears-1951}, we establish the validity of
Theorem \ref{carlitz-general-1}, as follows.

\begin{proof}
  On the one hand,  since $ \eta_a\eta_b\cdot b\eta_b=qb\eta_b\cdot \eta_a\eta_b$ ,
  by  \cite[Ex 1.35]{gasper-book},  we have
  \begin{align}
    (\eta_a\eta_b-b\eta_b)^n\cdot f(a) & =\sum_{k=0}^n\begin{bmatrix}
 n \\k
          \end{bmatrix}_q(-b\eta_b)^k\eta_a^{n-k}\eta_b^{n-k}\cdot f(a)\notag         \\
      & =\sum_{k=0}^n\begin{bmatrix}
 n \\k
          \end{bmatrix}_q(-b)^kq^{\binom{k}{2}}\eta_a^{n-k}\eta_b^{n}\cdot f(a)\notag \\
      & =\sum_{k=0}^n\begin{bmatrix}
 n \\k
          \end{bmatrix}_q(-b)^kq^{\binom{k}{2}}f(aq^{n-k}), \notag
  \end{align}
  on the other hand,  by  \cite[(3.5)]{sears-1951} or \cite[Ex 1.3(i)]{gasper-book} we know
  \begin{align}
    (\eta_a\eta_b-b\eta_b)^n\cdot f(a) & =(\eta_a-b)(\eta_a-bq)\dots(\eta_a-bq^{n-1})\eta_b^{n}\cdot f(a)\notag      \\
      & =\sum_{k=0}^n\begin{bmatrix}
 n \\k
          \end{bmatrix}_qb^k(b)_{n-k}\eta_a^{n-k}\Delta_a^k f(a), \notag
  \end{align}
  where
  $ \Delta_a^k=(\eta_a-1)(\eta_a-q)\dots(\eta_a-q^{k-1}), $
  thus, (\ref{carlitz-cor-1}) holds. Next we prove (\ref{carlitz-cor-11}),
similar to the proof of Wang \cite{wmj2013}, we can rewrite \eqref{carlitz-cor-1} as
    \begin{align}
      \sum_{k=0}^n\begin{bmatrix}
         n \\k
       \end{bmatrix}_q(-1)^kq^{\binom{k}{2}}f(aq^{n-k})\frac{b^k}{(b;q)_{\infty}}\!=\!\sum_{k=0}^n\begin{bmatrix}
n \\k
         \end{bmatrix}_q\eta_a^{k}\Delta_a^{n-k}f(a)\frac{b^{n-k}}{(bq^k;q)_{\infty}}. \notag
    \end{align}
    Now letting $b=bt$ and then multiplying the above equation by $(qt/z,qt;q)_{\infty}$, we obtain
    \begin{align}
      \sum_{k=0}^n\begin{bmatrix}
         n \\k
       \end{bmatrix}_q(-b)^kq^{\binom{k}{2}}f(aq^{n-k})
      \frac{(qt/z,qt;q)_{\infty}t^k}{(bt;q)_{\infty}}\!=\!\sum_{k=0}^n\begin{bmatrix}
      n \\k
    \end{bmatrix}_qb^{n-k}\eta_a^{k}\Delta_a^{n-k}f(a)
      \frac{(qt/z,qt;q)_{\infty}t^{n-k}}{(btq^k;q)_{\infty}}. \notag
    \end{align}
    Next, taking the $q$-integral on both sides of the above identity concerning variable $t$,
    \begin{align}
       & \sum_{k=0}^n\begin{bmatrix}
 n \\k
          \end{bmatrix}_q(-b)^kq^{\binom{k}{2}}f(aq^{n-k})\int_z^1
      \frac{(qt/z,qt;q)_{\infty}t^k}{(bt;q)_{\infty}}d_qt\notag\\
       & \!=\!\sum_{k=0}^n\begin{bmatrix}
      n \\k
    \end{bmatrix}_qb^{n-k}\eta_a^{k}\Delta_a^{n-k}f(a)\int_z^1
      \frac{(qt/z,qt;q)_{\infty}t^{n-k}}{(btq^k;q)_{\infty}}d_qt. \label{proof-1.2-1}
    \end{align}
    by \cite[(3.11) and (3.12)]{wmj2013}, we know
    \begin{align}
      \int_x^1 \frac{(q t / x, q t ; q)_{\infty} t^{n-k}}{(b t q^k ; q)_{\infty}} d_q t=\frac{(1-q)(q, q / x, x ; q)_{\infty}}{(b q^k x, b q^k ; q)_{\infty}} \varphi_{n-k}^{(b q^k)}(x | q), \notag
      \end{align}
      and
    \begin{align}
      \int_x^1 \frac{(q t / x, q t ; q)_{\infty} t^k}{(b t ; q)_{\infty}} d_q t=\frac{(1-q)(q, q / x, x ; q)_{\infty}}{(b x, a ; q)_{\infty}} \varphi_k^{(b)}(x | q),\notag
    \end{align}
 substituting above two identities into (\ref{proof-1.2-1}) yields (\ref{carlitz-cor-11}). Finally, we use the general Andrews-Askey integral given
 by Wang \cite{wmj2008} to prove (\ref{carlitz-cor-111}). multiply (\ref{carlitz-cor-1}) by $\frac{(qb/s,qb/t;q)_{\infty}}{(ub,vb;q)_{\infty}}$ we
 can obtain
 \begin{align}
  \sum_{k=0}^n\begin{bmatrix}
    n \\k
  \end{bmatrix}_q&\frac{(qb/s,qb/t;q)_{\infty}(b;q)_{k}b^{n\!-\!k}}{(ub,vb;q)_{\infty}}\eta_a^{k}\Delta_a^{n\!-\!k} f(a)\notag\\
  &\!=\!\sum_{k=0}^n\begin{bmatrix}
n \\k
\end{bmatrix}_q(-1)^kq^{\binom{k}{2}}\frac{(qb/s,qb/t;q)_{\infty}b^k}{(ub,vb;q)_{\infty}}f(aq^{n-k}),\notag
 \end{align}
use the $q$-integral on both sides of the above equation with respect to variable $b$, we have
\begin{align}
  \sum_{k=0}^n\begin{bmatrix}
    n \\k
  \end{bmatrix}_q&\int_{s}^{t}\frac{(qb/s,qb/t;q)_{\infty}(b;q)_{k}b^{n\!-\!k}}{(ub,vb;q)_{\infty}}d_qb\cdot\eta_a^{k}\Delta_a^{n\!-\!k} f(a)\notag\\
  &\!=\!\sum_{k=0}^n\begin{bmatrix}
n \\k
\end{bmatrix}_q(-1)^kq^{\binom{k}{2}}\int_{s}^{t}\frac{(qb/s,qb/t;q)_{\infty}b^k}{(ub,vb;q)_{\infty}}d_qb\cdot f(aq^{n-k}),\label{proof-1.2-2}
\end{align}
by \cite[(4.13) and (4.31)]{wmj2008}, we know
\begin{align}
  \int_{s}^{t}\frac{(qb/s,qb/t;q)_{\infty}b^k}{(ub,vb;q)_{\infty}}d_qb\!=\!
  \frac{t(1\!-\!q)(q,tq/s,s/t,uvst;q)_{\infty}}{u^k(us,ut,vs,vt;q)_{\infty}}\ 
  _3\phi_2\left(\begin{gathered}
    vs,vt,q^{\!-\!k}\\
    0,uvst
  \end{gathered};q,q\right),\notag
\end{align}
and 
\begin{align}
  \int_{s}^{t}\frac{(qb/s,qb/t;q)_{\infty}(b;q)_{k}b^{n\!-\!k}}{(ub,vb;q)_{\infty}}&d_qb\!=\!\frac{t(1\!-\!q)(1/u;q)_k(q,tq/s,s/t,uvst;q)_{\infty}}{v^{n\!-\!k}(us,ut,vs,vt;q)_{\infty}}\notag\\
  &\times\sum_{i=0}^{k}\frac{(q^{\!-\!k},us,ut;q)_iq^i}{(q,uq^{\!-\!k\!+\!1},uvst;q)_i} \  _3\phi_2\left(\begin{gathered}
    vs,vt,q^{k\!-\!n}\\
    0,uvstq^i
  \end{gathered};q,q\right),\notag
\end{align}
substituting the two above equation into (\ref{proof-1.2-2}) gives (\ref{carlitz-cor-111}).
\end{proof}

\subsection{The proof of Theorem \ref{carlitz-general-2}}
\begin{proof}
  Let $f(a)=\frac{(ay;q)_{\infty}}{(ax;q)_{\infty}}$, then
  $$f(aq^{n-k})=f(a)\frac{(ax;q)_{n-k}}{(ay;q)_{n-k}},$$
   substitution into the right
  side of (\ref{carlitz-cor-1}),  we have
  \begin{align}\label{proof-carlitz-cor-2-l}
    \sum_{k=0}^n\begin{bmatrix}
       n \\k
     \end{bmatrix}_q(-b)^kq^{\binom{k}{2}}f(aq^{n-k}) & =\sum_{k=0}^n\begin{bmatrix}
    n \\k
  \end{bmatrix}_q(-b)^kq^{\binom{k}{2}}\frac{(ayq^{n-k};q)_{\infty}}{(axq^{n-k};q)_{\infty}}\notag   \\
         & =\frac{(ay;q)_{\infty}}{(ax;q)_{\infty}}\sum_{k=0}^n\begin{bmatrix}
         n \\k
       \end{bmatrix}_q\frac{(ax;q)_{n-k}}{(ay;q)_{n-k}}(-b)^kq^{\binom{k}{2}},
  \end{align}
  by induction,  we know
  $$ \Delta_a^nf(a)=\frac{(ayq^n;q)_{\infty}P_n(x, y)(-a)^n}{(ax;q)_{\infty}}q^{\binom{n}{2}}, $$
  substituting into the left side of (\ref{carlitz-cor-1}),  we deduce that
  \begin{align}
     & \sum_{k=0}^n\begin{bmatrix}
          n \\k
        \end{bmatrix}_qb^k(b;q)_{n\!-\!k}\eta_a^{n-k}\Delta_a^kf(a)\notag       \\
     & =\sum_{k=0}^n\begin{bmatrix}
n \\k
         \end{bmatrix}_qb^k(b;q)_{n\!-\!k}\eta_{a}^{n\!-\!k}\frac{(ayq^k;q)_{\infty}P_k(x, y)(-a)^k}{(ax;q)_{\infty}}q^{\binom{k}{2}}\notag       \\
     & =\sum_{k=0}^n\begin{bmatrix}
n \\k
         \end{bmatrix}_qb^k(b;q)_{n\!-\!k}\frac{(ayq^n;q)_{\infty}P_k(x, y)(-aq^{n-k})^k}{(axq^{n\!-\!k};q)_{\infty}}q^{\binom{k}{2}}\notag       \\
     & =\frac{(ay;q)_{\infty}}{(ax;q)_{\infty}(ay;q)_n}\sum_{k=0}^n\begin{bmatrix}
   n \\k
 \end{bmatrix}_q(ax, b;q)_{n-k}P_k(x, y)(-ab)^kq^{\binom{k}{2}+k(n-k)}\notag    \\
     & =\frac{(ay;q)_{\infty}}{(ax;q)_{\infty}(ay;q)_n}\sum_{k=0}^n\begin{bmatrix}
   n \\k
 \end{bmatrix}_q(ax, b;q)_{k}P_{n-k}(x, y)(-ab)^{n-k}q^{\binom{n}{2}-\binom{k}{2}}, \notag
  \end{align}
  hence
  \begin{align}
    \frac{(ay;q)_{\infty}}{(ax;q)_{\infty}} & \sum_{k=0}^n\begin{bmatrix}
     n \\k
   \end{bmatrix}_q\frac{(ax;q)_{n-k}}{(ay;q)_{n-k}}(-b)^kq^{\binom{k}{2}}\notag      \\
& =\frac{(ay;q)_{\infty}}{(ax;q)_{\infty}(ay;q)_n}\sum_{k=0}^n\begin{bmatrix}
        n \\k
      \end{bmatrix}_q(ax, b;q)_{k}P_{n-k}(x, y)(-ab)^{n-k}q^{\binom{n}{2}-\binom{k}{2}},\notag
  \end{align}
  multiply both sides by $ \frac{(ax;q)_{\infty}(ay;q)_n}{(ay;q)_{\infty}}$  yields
  (\ref{carlitz-cor-2}). Similarly, let $f(a)=\frac{(ay;q)_{\infty}}{(ax;q)_{\infty}}$ in (\ref{carlitz-cor-11}) and (\ref{carlitz-cor-111}) we can drive (\ref{carlitz-cor-22}) and (\ref{carlitz-cor-222}).
\end{proof}
\subsection{The alternative proof of Theorem \ref{carlitz1974-them}}
\begin{proof}
  Let $ x=1, y=q$  in (\ref{carlitz-cor-2}) we obtain (\ref{Carlitz1974}).
\end{proof}

\section{The  proofs of Theorem \ref{app-lemma-1}-Theorem \ref{app-th2}}
\subsection{The proof of Theorem \ref{app-lemma-1}}

\begin{proof}
  Let $q \rightarrow q^2$ and $n \rightarrow n-1$ in (\ref{carlitz-m}). Then, the left side of (\ref{carlitz-m}) becomes
  \begin{align}\label{app-equ-1-proof-1}
    LHS & =\sum_{k=0}^{n-1}\frac{(a, b;q^2)_k}{(q^2;q^2)_k}
    \frac{(q^{2m};q^2)_{n\!-\!1\!-\!k}}{(q^{2};q^2)_{n\!-\!1\!-\!k}}(-ab)^{n\!-\!k\!-\!1}
    q^{(n\!-\!1)^2\!-\!(n\!-\!1)\!-\!k^2\!+\!k}\notag       \\
        & =\sum_{k=0}^{n-1}\frac{(a, b;q^2)_k}{(q^2;q^2)_k}
    \frac{(q^{2n\!-\!2k};q^2)_{m\!-\!1\!}}{(q^{2};q^2)_{m\!-\!1}}(-ab)^{n\!-\!k\!-\!1}
    q^{(n\!-\!1)(n\!-\!2)\!-\!k^2\!+\!k},
  \end{align}
  and the right side of   (\ref{carlitz-m}) becomes
  \begin{align}\label{app-equ-1-proof-2}
    RHS & =\sum_{k=0}^{n-1}\frac{(-b)^kq^{k^2-k}}{(q^2;q^2)_k(q^2;q^2)_{n-k-1}}
    \frac{(a;q^2)_{m+n-1}}{(aq^{2n-2k-2};q^2)_m}\notag        \\
        & =\sum_{k=\!-\!\frac{n\!-\!1}{2}}^{\frac{n\!-\!1}{2}}
    \frac{(-b)^{\frac{n-1}{2}\!+\!k}q^{(\frac{n-1}{2}\!+\!k)^2-(\frac{n-1}{2}\!+\!k)}}
    {(q^2;q^2)_{\frac{n-1}{2}\!+\!k}(q^2;q^2)_{\frac{n\!-\!1}{2}\!-\!k}}
    \frac{(a;q^2)_{m\!+\!n\!-\!1}}{(aq^{n\!-\!2k-1};q^2)_{m}}\notag      \\
        & =\sum_{k=\!-\!\frac{n\!-\!1}{2}}^{\frac{n\!-\!1}{2}}
    \frac{(q^{n\!-\!2k\!+\!1};q^2)_k(a;q^2)_{n\!+\!m\!-\!1}}
    {(q^2;q^2)_{\frac{n-1}{2}}^2(q^{n\!+\!1};q^2)_k(aq^{n\!-\!2k\!-\!1};q^2)_m}
    (-b)^{\frac{n-1}{2}\!+\!k}q^{(\frac{n-1}{2}\!+\!k)(\frac{n-3}{2}\!+\!k)}\notag  \\
        & =\frac{(q;q^2)_n}{(q^2;q^2)_{\frac{n\!-\!1}{2}}^2(1\!-\!q^n)}
    \frac{(a;q^2)_{n\!+\!m\!-\!1}}{(q;q^2)_n}(1-q^n)\notag    \\
        & \ \ \ \ \ \ \ \ \times\sum_{k=\!-\!\frac{n\!-\!1}{2}}^{\frac{n\!-\!1}{2}}
    \frac{(q^{n\!-\!2k\!+\!1};q^2)_k}
    {(q^{n\!+\!1};q^2)_k}\frac{(-b)^{\frac{n-1}{2}\!+\!k}q^{\frac{(n\!-\!1)(n\!-\!3)}{4}\!+\!k^2\!+\!(n\!-\!2)k}}
    {(aq^{n\!-\!2k\!-\!1};q^2)_m},
  \end{align}
  where
  $$ (q^{n\!-\!2k\!+\!1};q^2)_k=(-1)^kq^{k(n-k)}(q^{1-n};q^2)_k, $$
  and
  $$\frac{(q;q^2)_n}{(q^2;q^2)_{\frac{n-1}{2}}^2(1-q^n)}=\begin{bmatrix}
      n\!-\!1 \\\frac{n-1}{2}
    \end{bmatrix}_{q^2}\begin{bmatrix}
      2n\!-\!1 \\n\!-\!1
    \end{bmatrix}_q\frac{1}{(-q;q)_{n-1}^2}, $$
  substituting the above two equations into (\ref{app-equ-1-proof-2}) to obtain
  \begin{align}
    RHS & =\begin{bmatrix}
  n\!-\!1 \\\frac{n\!-\!1}{2}
\end{bmatrix}_{q^2}\begin{bmatrix}
          2n\!-\!1 \\n\!-\!1
        \end{bmatrix}_q\frac{(\!-\!1)^{\frac{n\!-\!1}{2}}q^{\frac{(n\!-\!1)(n\!-\!3)}{4}}}{(-q;q)_{n-1}^2}
    \frac{(a;q^2)_{n\!+\!m\!-\!1}}{(q;q^2)_n}\notag\\
        & \ \ \ \ \ \ \ \times(1-q^n)\sum_{k=\!-\!\frac{n\!-\!1}{2}}^{\frac{n\!-\!1}{2}}
    \frac{(q^{1-n};q^2)_k}{(q^{1+n};q^2)_k}\frac{b^{\frac{n-1}{2}\!+\!k}q^{2nk-2k}}
    {(aq^{n\!-\!2k\!-\!1};q^2)_m}, \notag
  \end{align}
  combining (\ref{app-equ-1-proof-1}) yields (\ref{app-lemma-equ-1}).
\end{proof}

\subsection{The proof of Theorem \ref{app-th1}}
\begin{proof}
  Let $ a\rightarrow q$ , $ b\rightarrow -q$  in (\ref{app-lemma-equ-1}),  we have
  \begin{align}\label{app-th1-proof-0}
    \sum_{k=0}^{n-1} & \frac{(q^2;q^4)_k}{(q^2;q^2)_k}(q^{2n\!-\!2k};q^2)_{m\!-\!1}
    q^{-k^2\!-\!k}\notag        \\
          & =\begin{bmatrix}
    n\!-\!1 \\\frac{n\!-\!1}{2}
  \end{bmatrix}_{q^2}\begin{bmatrix}
 2n\!-\!1 \\n\!-\!1
\end{bmatrix}_q\frac{(q^2;q^2)_{m\!-\!1}}{(-q;q)_{n\!-\!1}^2}q^{\!-\!\frac{(n\!-\!1)(3n\!+\!1)}{4}}
    \notag\\
          & \ \ \ \ \ \  \ \ \ \ \times (1\!-\!q^n)\!\sum_{k=\!-\!\frac{n\!-\!1}{2}}^{\frac{n\!-\!1}{2}}
    \frac{(q^{1\!-\!n};q^2)_k}{(q^{1\!+\!n};q^2)_k}
    \frac{(q^{2n\!+\!1};q^2)_{m\!-\!1}}{(q^{n\!-\!2k};q^2)_m}(\!-\!1)^kq^{2nk\!-\!k}.
  \end{align}
  On the one hand, on the left side of (\ref{app-th1-proof-0}),  when
  $ 0\leq k<m-1 $,  $ (q^{2n-2k};q^2)_{m-1}$ contains a factor
  $ (1-q^n) $,  due to $ 1-q^n \equiv0 \mod \Phi_n (q)$,  hence
  \begin{align}\label{app-th1-proof-1}
     & \sum_{k=0}^{n-1}
    \frac{(q^2;q^4)_k}{(q^2;q^2)_k}(q^{2n-2k};q^2)_{m\!-\!1}q^{\!-\!k^2\!-\!k}\notag \\
     & \equiv\sum_{k=m\!-\!1}^{n\!-\!1}\frac{(q^2;q^4)_k}
    {(q^2;q^2)_k}(q^{\!-\!2k};q^2)_{m\!-\!1}
    q^{\!-\!k^2\!-\!k}\mod \Phi_n(q)\notag\\
     & \equiv \sum_{k=0}^{n-m}\frac{(q^2;q^4)_{k\!+\!m\!-\!1}}
    {(q^2;q^2)_{k\!+\!m\!-\!1}}
    (q^{\!-\!2k\!-\!2m\!+\!2};q^2)_{m\!-\!1}
    q^{k\!+\!m\!-\!(k\!+\!m)^2}\mod \Phi_n(q),
  \end{align}
  where
  \begin{align}
    (q^{\!-\!2k\!-\!2m\!+\!2};q^2)_{m\!-\!1}=(-1)^{m-1}q^{-(m-1)(2k+m)}\frac{(q^2;q^2)_{m+k-1}}{(q^2;q^2)_{k}}, \notag
  \end{align}
  substituting the above equation into (\ref{app-th1-proof-1}) to obtain
  \begin{align}\label{app-th1-proof-2}
     & \sum_{k=0}^{n-1}
    \frac{(q^2;q^4)_k}{(q^2;q^2)_k}(q^{2n-2k};q^2)_{m\!-\!1}q^{\!-\!k^2\!-\!k}\notag \\
     & \equiv(\!-\!1)^{m\!-\!1}q^{\!-\!2m^2\!+\!2m}(q^2;q^4)_{m\!-\!1}
    \sum_{k=0}^{n-m}\frac{(q^{4m\!-\!2};q^4)_k}{(q^2;q^2)_k}q^{3k\!-\!k^2\!-\!4km}\mod \Phi_n(q).
  \end{align}
  On the other hand, in \cite[Th 1.2]{Panhao-2007} and \cite[(1.5)]{Ljx-Ph-Zy-2015}, Pan et al. provided a $q$-analogue of Morley's congruence as follows:
  \begin{align*}
    (-\!1)^{\frac{n-1}{2}}q^{\frac{n^2-1}{4}}\begin{bmatrix}
      n\!-\!1 \\\frac{n-1}{2}
    \end{bmatrix}_{q^2}\equiv (-q;q)_{n-1}^2-\frac{n^2-1}{24}(1-q^n)^2\mod \Phi_n(q)^3,
  \end{align*}
  uses the special case of the above $q$-congruence as follows:
  \begin{align}\label{equ-pan}
    \begin{bmatrix}
      n\!-\!1 \\\frac{n-1}{2}
    \end{bmatrix}_{q^2}\equiv(-1)^{\frac{n-1}{2}}q^{\frac{1-n^2}{4}}(-q;q)_{n-1}^2\mod \Phi_n(q)^2,
  \end{align}
  and the following $q$-congruence, as given by Guo in \cite[(3.1)]{GJW-2018},
  \begin{align}\label{equ-guo}
    \begin{bmatrix}
      2n\!-\!1 \\n\!-\!1
    \end{bmatrix}_q\equiv(-1)^{n-1}q^{\binom{n}{2}}\mod \Phi_n(q)^2,
  \end{align}
  we know that on the right side of (\ref{app-th1-proof-0}),  we can deduce that
  \begin{align}\label{app-th1-proof-3}
     & \begin{bmatrix}
         n\!-\!1 \\\frac{n\!-\!1}{2}
       \end{bmatrix}_{q^2}\begin{bmatrix}
      2n\!-\!1 \\n\!-\!1
    \end{bmatrix}_q\frac{(q^2;q^2)_{m\!-\!1}}{(-q;q)_{n\!-\!1}^2}q^{\!-\!\frac{(n\!-\!1)(3n\!+\!1)}{4}}
    \notag     \\
     & \ \ \ \ \ \  \ \ \ \ \times (1\!-\!q^n)\!\sum_{k=\!-\!\frac{n\!-\!1}{2}}^{\frac{n\!-\!1}{2}}
    \frac{(q^{1\!-\!n};q^2)_k}{(q^{1\!+\!n};q^2)_k}
    \frac{(q^{2n\!+\!1};q^2)_{m\!-\!1}}{(q^{n\!-\!2k};q^2)_m}(\!-\!1)^kq^{2nk\!-\!k}\notag   \\
     & \equiv\! (\!-\!1)^{\frac{n\!-\!1}{2}}q^{\!-\!\frac{n^2\!-\!1}{2}}(q^2;q^2)_{m\!-\!1}
    (1\!-\!q^n)\!\sum_{k=\!-\!\frac{n\!-\!1}{2}}^{\frac{n\!-\!1}{2}}
    \frac{(q;q^2)_{m-1}}{(q^{n\!-\!2k};q^2)_m}(\!-\!1)^kq^{\!-\!k}\mod \Phi_n(q),
  \end{align}
  in the summation formula to the right of the above $q$-congruence,
  only when $0\leq k\leq m\!-\!1$, $(q^{n\!-\!2k};q^2)_m$ contains the factor $(1-q^n)$. By
  combining this with $ q^n\equiv1\mod\Phi_n(q)$, we obtain:
  \begin{align}\label{app-th1-proof-4}
     & (1\!-\!q^n)\!\sum_{k=\!-\!\frac{n\!-\!1}{2}}^{\frac{n\!-\!1}{2}}
    \frac{(q;q^2)_{m-1}}{(q^{n\!-\!2k};q^2)_m}(\!-\!1)^kq^{\!-\!k}\notag \\
     & \equiv(1\!-\!q^n)\!\sum_{k=0}^{m-1}
    \frac{(q;q^2)_{m-1}}{(q^{n-2k};q^2)_{k}(1\!-\!q^n)(q^{n\!+\!2};q^2)_{m\!-\!1\!-\!k}}(\!-\!1)^kq^{\!-\!k}\mod\Phi_n(q)\notag \\
     & \equiv \sum_{k=0}^{m-1}
    \frac{(q;q^2)_{m-1}}{(q^{-2k};q^2)_{k}(q^{2};q^2)_{m\!-\!1\!-\!k}}(\!-\!1)^kq^{-\!k}\mod\Phi_n(q)\notag          \\
     & \equiv \sum_{k=0}^{m-1}
    \frac{(q;q^2)_{m-1}}{(q^{2};q^2)_{k}(q^{2};q^2)_{m\!-\!1\!-\!k}}q^{k^2}\mod\Phi_n(q)\notag \\
     & \equiv\frac{(q;q^2)_{m-1}}{(q^2;q^2)_{m-1}}\sum_{k=0}^{m-1}\begin{bmatrix}
  m\!-\!1 \\k
\end{bmatrix}_{q^2}q^{k^2}\mod\Phi_n(q),
  \end{align}
  by  \cite[Ex 1.2(vi)]{gasper-book},  we know that
  $$ (z;q)_{m-1}=\sum_{k=0}^{m-1}\begin{bmatrix}
      m\!-\!1 \\k
    \end{bmatrix}_q (-z)^kq^{\frac{k^2-k}{2}}, $$
  let $ q\rightarrow q^2$  and then let $ z\rightarrow -q$  in the above equation,  we obtain
  $$ (-q;q^2)_{m-1}=\sum_{k=0}^{m-1}\begin{bmatrix}
      m\!-\!1 \\k
    \end{bmatrix}_{q^2}q^{k^2}, $$
  substituting the above equation into (\ref{app-th1-proof-4}),  we have
  \begin{align}
    (1\!-\!q^n) & \!\sum_{k=\!-\!\frac{n\!-\!1}{2}}^{\frac{n\!-\!1}{2}}
    \frac{(q^{1\!-\!n};q^2)_k}{(q^{1\!+\!n};q^2)_k}
    \frac{(q^{2n\!+\!1};q^2)_{m-1}}{(q^{n\!-\!2k};q^2)_m}(\!-\!1)^kq^{2kn\!-\!k}
    \equiv\frac{(q^2;q^4)_{m-1}}{(q^2;q^2)_{m-1}}\mod\Phi_n(q), \notag
  \end{align}
  combining (\ref{app-th1-proof-2}), (\ref{app-th1-proof-3}) and (\ref{app-th1-proof-4}),
  we obtain (\ref{app-th1-equ}).
\end{proof}

\subsection{The proof of Theorem \ref{app-th2}}
\

\

Before proving Theorem \ref{app-th2},  we first prove the following lemma.
\begin{lemma}\label{app-th2-lamma2}
  Let $ n$  be a positive odd integer. Then
  \begin{align}
    \frac{(q^{2d+1};q^2)_{n}}{(q;q)_{n-1}}\equiv (1-q^n) \mod\Phi_n(q)^2.
  \end{align}
\end{lemma}
\begin{proof}
  Clearly,
  \begin{align}
    \frac{(q^{2d+1};q^2)_{n}}{(q;q)_{n-1}} & =
    \frac{(q^{2d+1};q^2)_{\frac{n\!-\!1}{2}\!-\!d}(1\!-\!q^n)(q^{n\!+\!2};q^2)_{\frac{n-1}{2}\!+\!d}}{(q;q)_{n-1}}\notag \\
&=(1\!-\!q^n)\frac{(q^{n-2};q^{-2})_{\frac{n\!-\!1}{2}\!-\!d}(q^{n\!+\!2};q^2)_{\frac{n-1}{2}\!+\!d}}{(q;q)_{n-1}}\notag     
  \end{align}
  since
   \begin{align*}
    (1-q^n)\equiv 0\mod\Phi_n(q), 
   \end{align*}
   and 
  \begin{align*}
    \frac{(q^{n-2};q^{-2})_{\frac{n\!-\!1}{2}\!-\!d}(q^{n\!+\!2};q^2)_{\frac{n-1}{2}\!+\!d}}{(q;q)_{n-1}}\equiv \frac{(q^{-2};q^{-2})_{\frac{n\!-\!1}{2}\!-\!d}(q^{2};q^2)_{\frac{n-1}{2}\!+\!d}}{(q;q)_{n-1}} \mod\Phi_n(q),
  \end{align*}
  hence
  \begin{align*}
    \frac{(q^{2d+1};q^2)_{n}}{(q;q)_{n-1}} &\equiv (1\!-\!q^n)\frac{(q^{-2};q^{-2})_{\frac{n\!-\!1}{2}\!-\!d}(q^{2};q^2)_{\frac{n-1}{2}\!+\!d}}{(q;q)_{n-1}} \mod\Phi_n(q)^2\notag \\
    & \equiv (1\!-\!q^n)\frac{(q^2;q^2)_{\frac{n-1}{2}}^2(\!-\!1)^{\frac{n\!-\!1}{2}\!-\!d}q^{-\frac{n^2-1}{4}-d^2}}{(q;q)_{n-1}}
        \frac{(q;q^2)_d}{(q^{\!-\!2d\!+\!1};q^2)_d}\mod\Phi_n(q)^2,
  \end{align*}
  where we have
  $$ (q^{\!-\!2d\!+\!1};q^2)_d=(-1)^dq^{-d^2}(q;q^2)_d,$$
  and by (\ref{equ-pan}) and \cite[(2.7)]{wc-nhx-2022}, we deduce that
  \begin{align}
    \frac{(q^{2d+1};q^2)_{n}}{(q;q)_{n-1}} & \equiv (1\!-\!q^n)\frac{(q^2;q^2)_{\frac{n-1}{2}}^2
    (\!-\!1)^{\frac{n\!-\!1}{2}}q^{-\frac{n^2-1}{4}}}{(q;q)_{n-1}}\mod\Phi_n(q)^2\notag    \\
& \equiv (1\!-\!q^n)(\!-\!1)^{\frac{n\!-\!1}{2}}q^{-\frac{n^2-1}{4}}(-q;q)_{n-1}
    \begin{bmatrix}
      n\!-\!1 \\\frac{n-1}{2}
    \end{bmatrix}_{q^2}^{-1}\mod\Phi_n(q)^2\notag         \\
& \equiv(1-q^n)\mod\Phi_n(q)^2.
  \end{align}
  Therefore, we complete the proof of Lemma \ref{app-th2-lamma2}.
\end{proof}
\begin{lemma}
  Let $ n$  be a positive odd integer, $ d \in\{-\frac{n-3}{2}, -\frac{n-5}{2}, \dots, \frac{n-5}{2},
    \frac{n-3}{2}\}$, and let
  $$\Psi_{n, d}(q)=\sum_{k=\!-\!\frac{n\!-\!1}{2}, k\neq d}^{\frac{n\!-\!1}{2}}
    \frac{(\!-\!1)^kq^{2dk\!-\!k}}{1\!-\!q^{2d-2k}}. $$
  Then
  \begin{align}\label{congru-lemma-eq-0}
    \Psi_{n, d}(q) & \equiv(-1)^{d+1}dq^{2d^2-d}+\mathbf{Sgn}(d)(\!-\!1)^{\frac{n\!-\!1}{2}}q^{2d^2\!+\!\frac{n\!-\!1}{2}}
    \sum_{j=1}^{|d|}\frac{q^{j\!-\!2dj}(1\!+\!q^{(4j\!-\!2)d})}{1\!-\!q^{4j\!-\!2}}\notag     \\
        & \ \ \ \ \ \ \ \times\{1\!+\!(\!-\!1)^{j\!+\!d}\!+\!q^{2j\!-\!1}[(\!-\!1)^{j\!+\!d}\!-\!1]\}\mod\Phi_n(q).
  \end{align}
\end{lemma}
\begin{proof}
  When $ d=0$,  similar to Guo's proof in  \cite{gjw2019},
  \begin{align}\label{proof-congru-lemma-d equiv 0}
    \Psi_{n, 0}(q) & =\sum_{k=\!-\!\frac{n\!-\!1}{2}, k\neq 0}^{\frac{n\!-\!1}{2}}
    \frac{(-1)^kq^{-k}}{1-q^{-2k}}
    =\sum_{k=1}^{\frac{n-1}{2}}\frac{(-1)^kq^k}{1-q^{2k}}+
    \sum_{k=1}^{\frac{n-1}{2}}\frac{(-1)^kq^{-k}}{1-q^{-2k}}\notag       \\
        & =\sum_{k=1}^{\frac{n-1}{2}}\frac{(-1)^k(q^k-q^k)}{1-q^{2k}}=0,
  \end{align}
  Therefore,  (\ref{congru-lemma-eq-0}) holds when $ d=0$. Generally,  if
  $ -\frac{n-1}{2}< d<\frac{n-1}{2}$  and $ d\neq 0$,  we have
  \begin{align}\label{proof-congru-lemma-1}
    \Psi_{n, d}(q) & =\sum_{k=\!-\!\frac{n\!-\!1}{2}, k\neq d}^{\frac{n\!-\!1}{2}}
    \frac{(-1)^kq^{2dk-k}}{1-q^{2d-2k}}
    =\sum_{k=-\frac{n-1}{2}}^{d-1}\frac{(\!-\!1)^kq^{2dk\!-\!k}}{1\!-\!q^{2d\!-\!2k}}
    +\sum_{k=d+1}^{\frac{n-1}{2}}\frac{(\!-\!1)^kq^{2dk\!-\!k}}{1\!-\!q^{2d\!-\!2k}}\notag  \\
        & =\sum_{k=1}^{\frac{n-1}{2}\!+\!d}\frac{(\!-\!1)^{k\!+\!d}q^{2d^2\!-\!d\!+\!k\!-\!2dk}}{1-q^{2k}}
    +\sum_{k=1}^{\frac{n-1}{2}\!-\!d}\frac{(\!-\!1)^{k\!+\!d}q^{2d^2\!-\!d\!-\!k\!+\!2dk}}{1-q^{-2k}}\notag       \\
        & =(\!-\!1)^dq^{2d^2\!-\!d}\left(\sum_{k=1}^{\frac{n-1}{2}\!+\!d}\frac{(\!-\!1)^kq^{k\!-\!2dk}}{1-q^{2k}}
    +\sum_{k=1}^{\frac{n-1}{2}\!-\!d}\frac{(\!-\!1)^kq^{2dk\!-\!k}}{1-q^{-2k}}\right).
  \end{align}
  If $ 0<d<\frac{n-1}{2}$,  then
  \begin{align}\label{proof-congru-lemma-2}
     & \sum_{k=1}^{\frac{n-1}{2}\!+\!d}\frac{(\!-\!1)^kq^{k\!-\!2dk}}{1-q^{2k}}
    +\sum_{k=1}^{\frac{n-1}{2}\!-\!d}\frac{(\!-\!1)^kq^{2dk\!-\!k}}{1-q^{-2k}}\notag   \\
     & =\sum_{k=1}^{\frac{n-1}{2}\!-\!d}(\!-\!1)^k\left(\frac{q^{k\!-\!2dk}}{1-q^{2k}}
    +\frac{q^{2dk\!-\!k}}{1-q^{-2k}}\right)\!+\!
    \sum_{k=\frac{n-1}{2}\!-\!d\!+\!1}^{\frac{n\!-\!1}{2}\!+\!d}\frac{(-1)^kq^{k\!-\!2dk}}{1-q^{2k}},
  \end{align}
  where
  \begin{align}\label{proof-congru-lemma-3}
     & \sum_{k=1}^{\frac{n-1}{2}\!-\!d}(\!-\!1)^k\left(\frac{q^{k\!-\!2dk}}{1-q^{2k}}
    +\frac{q^{2dk\!-\!k}}{1-q^{-2k}}\right)
    =\sum_{k=1}^{\frac{n-1}{2}\!-\!d}(-1)^k\frac{q^{k-2dk}-q^{2dk+k}}{1-q^{2k}}\notag   \\
     & =\sum_{k=1}^{\frac{n-1}{2}\!-\!d}(-1)^k\frac{q^{-2dk}-q^{2dk}}{q^{-k}-q^k}
    =\sum_{k=1}^{\frac{n-1}{2}\!-\!d}(-1)^k\sum_{j=1}^{d}(q^{(2j-1)k}+q^{-(2j-1)k})\notag          \\
     & =\sum_{j=1}^{d}\sum_{k=1}^{\frac{n-1}{2}\!-\!d}[(-1)^kq^{(2j-1)k}+(-1)^kq^{-(2j-1)k}]\notag \\
     & =\sum_{j=1}^{d}\left(\frac{\!-\!q^{2j\!-\!1}\!+\!(\!-\!1)^{\frac{n\!-\!1}{2}\!-\!d}q^{(2j\!-\!1)(\frac{n-1}{2}-d+1)}}{1+q^{2j-1}}
    \!+\!\frac{\!-\!q^{\!-\!2j\!+\!1}\!+\!(\!-\!1)^{\frac{n\!-\!1}{2}\!-\!d}q^{\!-\!(2j\!-\!1)(\frac{n\!-\!1}{2}\!-\!d\!+\!1)}}{1+q^{-2j+1}}\right)\notag \\
     & \equiv -d+\sum_{j=1}^{d}\frac{(\!-\!1)^{\frac{n\!-\!1}{2}\!+\!d}q^{\frac{n\!-\!1}{2}\!+\!j}
    (q^{d\!-\!2dj}\!+\!q^{-d\!+\!2dj})}{1+q^{2j-1}}\mod \Phi_n(q),
  \end{align}
  and
  \begin{align}\label{proof-congru-lemma-4}
     & \sum_{k=\frac{n-1}{2}\!-\!d\!+\!1}^{\frac{n\!-\!1}{2}\!+\!d}\frac{(-1)^kq^{k\!-\!2dk}}{1-q^{2k}}
    =\sum_{k=\frac{n\!-\!1}{2}\!-\!d\!+\!1}^{\frac{n-1}{2}}\frac{(-1)^kq^{k\!-\!2dk}}{1-q^{2k}}
    \!+\!\sum_{k=\frac{n\!-\!1}{2}\!+\!1}^{\frac{n\!-\!1}{2}\!+\!d}\frac{(-1)^kq^{k\!-\!2dk}}{1-q^{2k}}\notag     \\
     & \equiv \sum_{k=1}^{d}\frac{(\!-\!1)^{k+\frac{n\!-\!1}{2}\!-\!d}q^{(k\!+\!\frac{n\!-\!1}{2}\!-\!d)(1\!-\!2d)}}{1-q^{2k-2d-1}}
    \!+\!\sum_{k=1}^{d}\frac{(\!-\!1)^{k\!+\!\frac{n\!-\!1}{2}}q^{(k\!+\!\frac{n\!-\!1}{2})(1\!-\!2d)}}{1-q^{2k-1}}\mod \Phi_n(q)\notag \\
     & \equiv \sum_{k=1}^{d}\frac{(\!-\!1)^{k\!+\!\frac{n\!-\!1}{2}\!+\!1}q^{(\frac{n\!-\!1}{2}\!-\!k\!+\!1)(1-2d)}}{1-q^{-(2k-1)}}
    \!+\!\sum_{k=1}^{d}\frac{(\!-\!1)^{k\!+\!\frac{n\!-\!1}{2}}q^{\frac{n\!-\!1}{2}\!+\!k-\!2dk\!+\!d}}{1-q^{2k-1}}\mod \Phi_n(q)\notag \\
     & \equiv\sum_{k=1}^{d}\frac{(\!-\!q)^{\frac{n-1}{2}\!+\!k}(q^{2dk\!-\!d}\!+\!q^{\!-\!2dk\!+\!d})}{1\!-\!q^{2k-1}}\mod \Phi_n(q),
  \end{align}

  combining (\ref{proof-congru-lemma-1}, \ref{proof-congru-lemma-2},
  \ref{proof-congru-lemma-3}, \ref{proof-congru-lemma-4}),  we know that
  when $ 0<d<\frac{n-1}{2}$,
  \begin{align}\label{proof-congru-lemma-d more than 0}
    \Psi_{n, d}(q)
     & \equiv(-1)^{d+1}dq^{2d^2-d}+(\!-\!1)^{\frac{n\!-\!1}{2}}q^{2d^2+\frac{n-1}{2}}
    \sum_{j=1}^d\frac{q^{j\!-\!2dj}(1\!+\!q^{(4j\!-\!2)d})}{1\!-\!q^{4j\!-\!2}}\notag      \\
     & \ \ \ \ \ \ \ \times\{1\!+\!(\!-\!1)^{j\!+\!d}\!+\!q^{2j\!-\!1}[(\!-\!1)^{j\!+\!d}\!-\!1]\}\mod\Phi_n(q),
  \end{align}
  therefore,  (\ref{congru-lemma-eq-0}) holds when $ 0<d<\frac{n-1}{2}$.
  When $ -\frac{n-1}{2}<d<0$,  make $ l=-d$,  then
  \begin{align}
     & \sum_{k=1}^{\frac{n-1}{2}\!+\!d}\frac{(\!-\!1)^kq^{k\!-\!2dk}}{1-q^{2k}}
    +\sum_{k=1}^{\frac{n-1}{2}\!-\!d}\frac{(\!-\!1)^kq^{2dk\!-\!k}}{1-q^{-2k}}\notag   \\
     & =\sum_{k=1}^{\frac{n-1}{2}\!-\!l}(\!-\!1)^k\left(\frac{q^{k\!+\!2lk}}{1-q^{2k}}
    +\frac{q^{-2lk\!-\!k}}{1-q^{-2k}}\right)\!+\!
    \sum_{k=\frac{n-1}{2}\!-\!l\!+\!1}^{\frac{n\!-\!1}{2}\!+\!l}\frac{(-1)^kq^{-k\!-\!2lk}}{1-q^{-2k}}, \notag
  \end{align}
  similar to (\ref{proof-congru-lemma-3}, \ref{proof-congru-lemma-4}),  we have
  \begin{align}
    \sum_{k=1}^{\frac{n-1}{2}\!-\!l}(\!-\!1)^k & \left(\frac{q^{k\!+\!2lk}}{1\!-\!q^{2k}}
    \!+\!\frac{q^{\!-\!2lk\!-\!k}}{1\!-\!q^{-2k}}\right)
    \equiv\! l\!-\!\sum_{j=1}^{l}\frac{(\!-\!1)^{\frac{n\!-\!1}{2}\!+\!l}q^{\frac{n\!-\!1}{2}\!+\!j}
    (q^{l\!-\!2lj}\!+\!q^{-l\!+\!2lj})}{1+q^{2j-1}}\mod \Phi_n(q), \notag
  \end{align}
  and
  \begin{align}
    \sum_{k=\frac{n-1}{2}\!-\!l\!+\!1}^{\frac{n\!-\!1}{2}\!+\!l}
    \frac{(-1)^kq^{-k\!-\!2lk}}{1-q^{-2k}}\equiv -\sum_{j=1}^{l}
    \frac{(-q)^{\frac{n-1}{2}\!+\!j}(q^{l\!-\!2lj}\!+\!q^{-l\!+\!2lj})}{1-q^{2j-1}}\mod \Phi_n(q), \notag
  \end{align}
  hence
  \begin{align} \label{proof-congru-lemma-d less than 0}
    \Psi_{n, d}(q)
     & \equiv(-1)^{d+1}dq^{2d^2-d}\!-\!(\!-\!1)^{\frac{n\!-\!1}{2}}q^{2d^2+\frac{n-1}{2}}
    \sum_{j=1}^{-d}\frac{q^{j\!-\!2dj}(1\!+\!q^{(4j\!-\!2)d})}{1\!-\!q^{4j\!-\!2}}\notag   \\
     & \ \ \ \ \ \ \ \times\{1\!+\!(\!-\!1)^{j\!+\!d}\!+\!q^{2j\!-\!1}[(\!-\!1)^{j\!+\!d}\!-\!1]\}\mod\Phi_n(q).
  \end{align}
  when $ -\frac{n-1}{2}<d<0$, Thus,  (\ref{congru-lemma-eq-0}) also holds when $ -\frac{n-1}{2}<d<0 $. Combining
  (\ref{proof-congru-lemma-d equiv 0}),
  (\ref{proof-congru-lemma-d more than 0}) and (\ref{proof-congru-lemma-d less than 0})
  we obtain (\ref{congru-lemma-eq-0}).
\end{proof}

Now, let's proceed to prove Theorem \ref{app-th2}.
\begin{proof}
  Let $ m=1$ , $ a=q^{2d+1}$ , $ b=-q^{2d+1}$  in (\ref{app-lemma-equ-1}).  We deduce that
  \begin{align}
    \sum_{k=0}^{n-1}\frac{(q^{4d\!+\!2};q^4)_k}{(q^2;q^2)_k} & q^{\!-\!4dk\!-\!k^2\!-\!k}
    =\begin{bmatrix}
       n\!-\!1 \\\frac{n\!-\!1}{2}
     \end{bmatrix}_{q^2}\begin{bmatrix}
    2n\!-\!1 \\n\!-\!1
  \end{bmatrix}_q\frac{(q^{2d+1};q^2)_nq^{-\frac{(n\!-\!1)(12d\!+\!3n\!+\!1)}{4}}}
    {(-q;q)_{n-1}^2(q;q^2)_n}\notag    \\
      & \times(1-q^n)\sum_{k=\!-\!\frac{n\!-\!1}{2}}^{\frac{n\!-\!1}{2}}
    \frac{(q^{1-n};q^2)_k}{(q^{1+n};q^2)_k}
    \frac{(-1)^kq^{(2n+2d)k-k}}{1-q^{n+2d-2k}}, \notag
  \end{align}
  since
  $$ \frac{(q^{2d+1};q^2)_n}{(q;q^2)_n}=\frac{(q;q^2)_{n+d}}{(q;q^2)_d(q;q^2)_n}
    =\frac{(q^{2n+1};q^2)_d}{(q;q^2)_d},$$
  using (\ref{equ-pan}, \ref{equ-guo}) we know that
  \begin{align}\label{proof-congru-theorem-1}
    \sum_{k=0}^{n-1}\frac{(q^{4d\!+\!2};q^4)_k}{(q^2;q^2)_k} & q^{\!-\!4dk\!-\!k^2\!-\!k}
    \equiv (-1)^{\frac{n-1}{2}}q^{\!-\!\frac{n^2\!-\!1}{2}\!+\!3d\!-\!3dn}
    \frac{(q^{2n+1};q^2)_d}{(q;q^2)_d}\notag      \\
      & \times(1-q^n)\sum_{k=\!-\!\frac{n\!-\!1}{2}}^{\frac{n\!-\!1}{2}}
    \frac{(q^{1-n};q^2)_k}{(q^{1+n};q^2)_k}
    \frac{(-1)^kq^{(2n+2d)k-k}}{1-q^{n+2d-2k}}\mod\Phi_n(q)^2\notag      \\
      & \equiv(-1)^{\frac{n-1}{2}}q^{\!-\!\frac{n^2\!-\!1}{2}\!+\!3d}
    \{(1\!-\!q^n)\Psi_{n, d}(q)\!+\!\notag        \\
      & \ \ \frac{(q^{2n\!+\!1};q^2)_d}{(q;q^2)_d}
    \frac{(q^{1\!-\!n};q^2)_d}{(q^{1\!+\!n};q^2)_d}(\!-\!1)^dq^{\!-\!dn\!+\!2d^2\!-\!d}\}\mod\Phi_n(q)^2.
  \end{align}
  If $ 0\leq d<\frac{n-1}{2}$,  then
  \begin{align}
    \frac{(q^{2n+1};q^2)_d}{(q;q^2)_d}\frac{(q^{1-n};q^2)_d}{(q^{1+n};q^2)_d} &
    =\prod_{j=0}^{d-1}\frac{(1-q^{2n+1+2j})(1-q^{1-n+2j})}{(1-q^{1+2j})(1-q^{n+1+2j})}\notag    \\
& =\prod_{j=0}^{d-1}\frac{-q^{1-n+2j}(1+q^{3n})+1+q^{n+2+4j}}{(1-q^{1+2j})(1-q^{n+1+2j})}\notag \\
& =\prod_{j=0}^{d-1}\frac{\!-\!q^{1\!-\!n\!+\!2j}(1\!+\!q^{n})(1\!-\!q^n\!+\!q^{2n})\!+\!1\!+\!q^{n\!+\!2\!+\!4j}}{(1-q^{1+2j})(1-q^{n+1+2j})}, \notag
  \end{align}
  from the proof of Theorem 1.2 in  \cite{Liujicai-petrov-2020}
  \begin{align}\label{equ-liujicai}
    1-q^{2n}=(1+q^n)(1-q^n)\equiv2(1-q^n)\mod \Phi_n(q)^2,
  \end{align}
  we know that
  $$ 1-q^n+q^{2n}\equiv q^n\mod \Phi_n(q)^2, $$
  hence
  \begin{align}
    \frac{(q^{2n+1};q^2)_d}{(q;q^2)_d}\frac{(q^{1-n};q^2)_d}{(q^{1+n};q^2)_d}
     & = \prod_{j=0}^{d-1}\frac{\!-\!q^{1\!-\!n\!+\!2j}(1\!+\!q^{n})(1\!-\!q^n\!+\!q^{2n})\!+\!1\!+\!q^{n\!+\!2\!+\!4j}}{(1-q^{1+2j})(1-q^{n+1+2j})}\notag \\
     & \equiv \prod_{j=0}^{d-1}\frac{\!-\!q^{1\!+\!2j}(1\!+\!q^{n})\!+\!1\!+\!q^{n\!+\!2\!+\!4j}}{(1-q^{1+2j})(1-q^{n+1+2j})}\mod \Phi_n(q)^2\notag        \\
     & \equiv \prod_{j=0}^{d-1}\frac{(1-q^{1+2j})(1-q^{n+1+2j})}{(1-q^{1+2j})(1-q^{n+1+2j})}\mod \Phi_n(q)^2\notag        \\
     & \equiv 1\mod \Phi_n(q)^2, \notag
  \end{align}
  Similarly,  if  $-\frac{n-1}{2}<d<0$,  then
  \begin{align}
    \frac{(q^{2n\!+\!1};q^2)_{d}}{(q;q^2)_{d}}\frac{(q^{1\!-\!n};q^2)_{d}}{(q^{1\!+\!n};q^2)_{d}}
    =\frac{(q;q^2)_{-d}(q^{1\!-\!n};q^2)_{-d}}{(q^{1\!-\!2n};q^2)_{-d}(q^{1\!+\!n};q^2)_{-d}}\equiv 1\mod\Phi_n(q)^2, \notag
  \end{align}
  combining (\ref{congru-lemma-eq-0}),  the formula for the braces in
  the final equation of (\ref{proof-congru-theorem-1}) becomes
  \begin{align}\label{proof-congru-theorem-3}
     & (1\!-\!q^n)\Psi_{n, d}(q)+\frac{(q^{2n+1};q^2)_d}{(q;q^2)_d}
    \frac{(q^{1-n};q^2)_d}{(q^{1+n};q^2)_d}(-1)^dq^{\!-\!dn\!+\!2d^2\!-\!d}\notag         \\
     & \equiv (1\!-\!q^n)\Psi_{n, d}(q)\!+\!(\!-\!1)^dq^{-dn\!+\!2d^2\!-\!d}\mod \Phi_n(q)^2\notag   \\
     & \equiv (1\!-\!q^n)(-1)^{d\!+\!1}dq^{2d^2\!-\!d}\!+\!(\!-\!1)^dq^{-dn\!+\!2d^2\!-\!d}\!+\!
    \mathbf{Sgn}(d)(1\!-\!q^n)(\!-\!1)^{\frac{n\!-\!1}{2}}q^{2d^2\!+\!\frac{n\!-\!1}{2}}\times\notag \\
     & \ \ \ \sum_{j=1}^{|d|}\frac{q^{j\!-\!2dj}(1\!+\!q^{(4j\!-\!2)d})}{1\!-\!q^{4j\!-\!2}}
    \{1\!+\!(\!-\!1)^{j\!+\!d}\!+\!q^{2j\!-\!1}[(\!-\!1)^{j\!+\!d}\!-\!1]\}\mod\Phi_n(q)^2\notag     \\
     & \equiv(\!-\!1)^dq^{2d^2\!-\!d}(q^{\!-\!dn}\!-\!d\!+\!dq^{n})
    \!+\! \mathbf{Sgn}(d)(1\!-\!q^n)(\!-\!1)^{\frac{n\!-\!1}{2}}q^{2d^2+\frac{n-1}{2}}\times\notag   \\
     & \ \ \ \sum_{j=1}^d\frac{q^{j\!-\!2dj}(1\!+\!q^{(4j\!-\!2)d})}{1\!-\!q^{4j\!-\!2}}
    \{1\!+\!(\!-\!1)^{j\!+\!d}\!+\!q^{2j\!-\!1}[(\!-\!1)^{j\!+\!d}\!-\!1]\}\mod\Phi_n(q)^2,
  \end{align}
  by (\ref{equ-liujicai}),  we know that
  \begin{align}
    (1\!-\!q^{-dn})\!=\!(1\!-\!q^{\!-\!n})(1\!+\!q^{\!-\!n}\!+\!\dots\!+\!q^{(\!-\!d\!+\!1)n})
    \equiv d(1\!-\!q^{\!-\!n})\equiv\!-\!d(1\!-\!q^n)\mod\Phi_n(q)^2, \notag
  \end{align}
  so that
  \begin{align}
    q^{-dn}-d(1-q^n)\equiv1\mod\Phi_n(q)^2,
  \end{align}
  thus,  the first term in the final form of  (\ref{proof-congru-theorem-3}) is
  \begin{align}\label{proof-congru-theorem-4}
     & (\!-\!1)^dq^{2d^2\!-\!d}(q^{\!-\!dn}\!-\!d\!+\!dq^{n})\equiv(\!-\!1)^dq^{2d^2\!-\!d}\mod\Phi_n(q)^2,
  \end{align}
  combining (\ref{proof-congru-theorem-1}), (\ref{proof-congru-theorem-3})
  and (\ref{proof-congru-theorem-4}) we have
  \begin{align}\label{proof-congru-theorem-5}
    \sum_{k=0}^{n-1} & \frac{(q^{4d\!+\!2};q^4)_k}{(q^2;q^2)_k}q^{\!-\!4dk\!-\!k^2\!-\!k}
    \!\equiv\! (\!-\!1)^{\frac{n\!-\!1}{2}\!+\!d}q^{\!-\!\frac{n^2\!-\!1}{2}\!+\!2d^2\!+\!2d}\!+\!
    \mathbf{Sgn}(d)(1\!-\!q^n)q^{\!-\!\frac{n(n\!-\!1)}{2}\!+\!3d\!+\!2d^2}\times\notag     \\
          & \sum_{j=1}^{|d|}\frac{q^{j\!-\!2dj}(1\!+\!q^{(4j\!-\!2)d})}{1\!-\!q^{4j\!-\!2}}
    \{1\!+\!(\!-\!1)^{j\!+\!d}\!+\!q^{2j\!-\!1}[(\!-\!1)^{j\!+\!d}\!-\!1]\}\mod\Phi_n(q)^2\notag       \\
          & \equiv (\!-\!1)^{\frac{n\!-\!1}{2}\!+\!d}q^{\!-\!\frac{n^2\!-\!1}{2}\!+\!2d^2\!+\!2d}\!+\!
    \mathbf{Sgn}(d)(1\!-\!q^n)q^{2d^2\!+\!3d}\times\notag  \\
          & \ \ \ \sum_{j=1}^{|d|}\frac{q^{j\!-\!2dj}(1\!+\!q^{(4j\!-\!2)d})}{1\!-\!q^{4j\!-\!2}}
    \{1\!+\!(\!-\!1)^{j\!+\!d}\!+\!q^{2j\!-\!1}[(\!-\!1)^{j\!+\!d}\!-\!1]\}\mod\Phi_n(q)^2.
  \end{align}
  Therefore, Theorem \ref{app-th2} holds.
\end{proof}

\section{Acknowledgment}
The author was partially supported by the National Natural Science Foundation of China (Grant 12371328). The author would like to express gratitude to Professor Zhiguo Liu for his invaluable technical and material support during the writing of this manuscript. Special thanks are extended to Jing Gu for carefully reviewing the initial draft and providing valuable insights. The author also acknowledges Feng Liu and Deliang Wei for their assistance during the writing process.

\end{document}